\documentclass[11pt,letterpaper]{article}
\usepackage[T1]{fontenc}
\usepackage[utf8]{inputenc}
\usepackage[affil-it]{authblk}
\usepackage[misc]{ifsym}
\usepackage{dsfont}
\usepackage{mathrsfs}
\usepackage{amsmath,amssymb}
\usepackage{amsfonts}

\usepackage{amsthm}
\usepackage{graphicx}
\usepackage{subfigure}
\usepackage{xcolor}

\usepackage{bbm}
\usepackage{mathrsfs}
\usepackage{txfonts}
\usepackage{color}
\usepackage{pict2e}
\usepackage{float}
\usepackage{palatino,epsfig,latexsym}
\usepackage{epsf,xy,epic,amscd}
\usepackage{lineno}

\usepackage{verbatim}
\usepackage{bibentry}

\theoremstyle{plain}
\newtheorem{theorem}{Theorem}[section]
\newtheorem{lemma}[theorem]{Lemma}

\newtheorem{proposition}[theorem]{Proposition}

\newtheorem{conjecture}[theorem]{Conjecture}

\newtheorem{Bounded Diameter Lemma}[theorem]{Bounded Diameter Lemma}
\theoremstyle{definition}
\newtheorem{definition}[theorem]{Definition}

\newtheorem{remark}[theorem]{Remark}

\newcommand{\Hmm}[1]{\leavevmode{\marginpar{\tiny%
$\hbox to 0mm{\hspace*{-0.5mm}$\leftarrow$\hss}%
\vcenter{\vrule depth 0.1mm height 0.1mm width \the\marginparwidth}%
\hbox to
0mm{\hss$\rightarrow$\hspace*{-0.5mm}}$\\\relax\raggedright #1}}}

\hfuzz=\maxdimen
\DeclareFixedFont{\Acknowledgment}{OT1}{cmr}{bx}{n}{14pt}
\begin{document}
%\pagewiselinenumbers
\numberwithin{equation}{section}

\title{Generalizing Andreev's Theorem via circle patterns}

\author{
\textsc{Ze Zhou} \\
\normalsize School of Mathematical Sciences, Shenzhen University,  Shenzhen 518060, China \\
\normalsize{zhouze@szu.edu.cn}
}

\date{}

\maketitle

\begin{abstract}
In this paper we derive an extended Circle Pattern Theorem that allows obtuse overlap angles. As a consequence, we characterize a subclass of compact convex hyperbolic polyhedra with possibly obtuse dihedral angles and thus generalize Andreev's Theorem. For proofs of these results, we establish a discrete analog of the "weak solution/regularity theory" scheme.


\end{abstract}

%\tableofcontents

%\setcounter{section}{-1}

\section{Introduction}
A circle pattern $\mathcal P$ on the Riemann sphere $\hat{\mathbb C}$ is a collection of closed disks. The contact graph $G(\mathcal P)$ of $\mathcal P$ is the graph having a vertex for each disk, and having an edge between the vertices $u$ and $w$ if the corresponding disks $D_u, D_w$ intersect each other. For each edge $e=[u,w]$ of $G(\mathcal P)$, we have the overlap angle $\Theta(e)\in[0,\pi)$.  We refer the readers to Stephenson's monograph~\cite{Stephenson} for basic background on circle patterns.

Let us consider the following questions: Given a graph $G$ and a function $\Theta:E\to[0,\pi)$ defined on the set of edges, is there a circle pattern with contact graph isomorphic to $G$ and overlap angles given by $\Theta$? If so, to what extent is the pattern unique? This problem is well posed under the condition that $0\leq\Theta\leq \pi/2$. In this case, by Thurston's interpretation~\cite{Thurston} of Andreev's Theorem~\cite{Andreev} (see also~\cite{Roeder-Hubbard-Dunbar}), the problem has a complete answer when $G$ is the $1$-skeleton of a triangulation of the sphere. Mention that the special case of vanishing angles was due to Koebe~\cite{Koebe}. The result is now known as the Koebe-Andreev-Thurston Theorem (or KAT Theorem for short), which plays a fundamental role in approximating conformal maps~\cite{Rodin-Sullivan} and hyperbolization of Haken 3-manifolds~\cite{Otal}. For an account of history of the KAT Theorem together with an extensive list of references, see the recent article of Philip Bowers~\cite{PBowers}.

During the past decades, there was fruitful research relating to KAT Theorem. These include the works of Marden-Rodin~\cite{Marden-Rodin}, Beardon-Stephenson~\cite{Beardon-Stephenson},  Colin de Verdi$\mathrm{\grave{e}}$re~\cite{Colin}, Bowers-Stephenson~\cite{Bowers-Stephenson}, Chow-Luo~\cite{Chow-Luo}, Bobenko-Springborn~\cite{Bobenko-Springborn}, Connelly-Gortler~\cite{Connelly-Gortler}, Bowers~\cite{Bowers} and many others.

For circle patterns with possibly obtuse overlap angles, whether similar results hold is a longstanding problem (see, e.g.,~\cite{Bowers-Stephenson,He}). Rivin-Hodgson~\cite{Rivin-Hodgson} described all compact convex hyperbolic polyhedra in terms of a generalized Gauss map, which was applied to deduction of Andreev's Theorem by Hodgson~\cite{Hodgson}. Yet this characterization fails to determine combinatorics in more general situations. Rivin~\cite{Rivin1, Rivin2, Rivin3}, Bao-Bonahon~\cite{Bao-Bonahon}, Bobenko-Springborn~\cite{Bobenko-Springborn} and Schlenker~\cite{Schlenker}  later made significant breakthroughs on ideal and hyperideal settings. Recently, Ge-Hua-Zhou~\cite{Ge-Hua-Zhou} and Jiang-Luo-Zhou~\cite{Jiang-Luo-Zhou} derived several generalizations of the Marden-Rodin Theorem~\cite{Marden-Rodin}, while their results gave only "weak solutions" that did not rule out extraneous overlaps. To the best of our knowledge, so far a fully direct generalization of the KAT Theorem that allows obtuse overlap angles remains open.

In this paper we shall attack the problem. We say a circle pattern $\mathcal P=\{D(v)\}_{v\in V}$ on the Riemann sphere  $\hat{\mathbb C}$ is \textbf{irreducible} if $\cup_{v\in A}D(v)\subsetneq \hat{\mathbb C}$ for every strict subset $A\subsetneq V$. Our main result is as follows.

\begin{theorem}\label{T-1-1}
Let $T$ be a triangulation of the sphere with more than four vertices. Suppose that $\Theta:E\to (0,\pi)$ is a function (where $E$ denotes the set of edges of $T$) satisfying the conditions below:
\begin{itemize}
\item[$\mathrm{\mathbf{(c1)}}$] If $e_1,e_2$ form a simple arc with non-adjacent endpoints, then $\Theta(e_1)+\Theta(e_2)\leq\pi$, and one of the inequalities is strict when $T$ is the boundary of a double tetrahedron.
\item[$\mathrm{\mathbf{(c2)}}$] If $e_1,e_2,e_3$ form the boundary of a triangle of $T$, then $\sum_{\mu=1}^3\Theta(e_\mu)>\pi$, and $\Theta(e_1)+\Theta(e_2)<\Theta(e_3)+\pi$, $\Theta(e_2)+\Theta(e_3)<\Theta(e_1)+\pi$, $\Theta(e_3)+\Theta(e_1)<\Theta(e_2)+\pi$.
\item[$\mathrm{\mathbf{(c3)}}$] If $e_1,e_2,e_3$ form a simple closed curve separating the vertices of $T$, then $\sum_{\mu=1}^3\Theta(e_\mu)<\pi$.
\item[$\mathrm{\mathbf{(c4)}}$] If $e_1,e_2,e_3,e_4$ form a simple closed curve separating the vertices of $T$, then $\sum_{\mu=1}^4\Theta(e_\mu)<2\pi$.
\end{itemize}
Then there exists an \textbf{irreducible} circle pattern $\mathcal P$ on the Riemann sphere $\hat{\mathbb C}$ with contact graph isomorphic to the $1$-skeleton of $T$ and overlap angles given by $\Theta$. Furthermore, $\mathcal P$ is unique up to linear and anti-linear fractional maps.
\end{theorem}

\begin{remark}
Condition $\mathrm{\mathbf{(c1)}}$ is motivated by an angle relation (see Lemma~\ref{L-4-4}) and is mainly applied to avoiding extraneous overlaps. Condition $\mathrm{\mathbf{(c2)}}$ is related to inner angles of spherical triangles. Remind that $\mathrm{\mathbf{(c1)}}, \mathrm{\mathbf{(c2)}}$ exclude some extreme cases noticed by
Bowers-Stephenson~\cite[p.\,214]{Bowers-Stephenson}. What is more, by an induction argument,  $\mathrm{\mathbf{(c1)}}-\mathrm{\mathbf{(c4)}}$ imply $\sum_{\mu=1}^k \Theta(e_\mu)<(k-2)\pi$  for $k\geq 5$ if $e_1,e_2,\cdots,e_k$ form a simple closed curve.
\end{remark}

In light of Thurston's interpretation~\cite[Chap. 13]{Thurston}, Theorem~\ref{T-1-1} leads to the following generalization of Andreev's Theorem~\cite{Andreev}. Mention that the rigidity part has already been proved by Rivin-Hodgson~\cite[Corollary 4.6]{Rivin-Hodgson}. For an abstract polyhedron $P$, recall that a prismatic $k$-circuit $\Gamma$ is a simple closed curve formed of $k$ edges of the dual complex $P^\ast$ such that all of the endpoints of the edges of $P$ intersected by $\Gamma$ are distinct.

\begin{theorem}\label{T-1-3}
Let $P$ be an abstract trivalent polyhedron with more than four faces. Suppose that $\Theta:\mathcal E\to (0,\pi)$ is a function (where $\mathcal E$ denotes the set of edges of $P$) such that the following conditions hold:
\begin{itemize}
\item[$\mathrm{\mathbf{(a1)}}$]Whenever $\Gamma$ is a simple arc formed by two edges of $P^\ast$ such that $\Gamma$ has non-adjacent endpoints in $P^\ast$ and intersects distinct edges $e_1,e_2$ of $P$, then $\Theta(e_1)+\Theta(e_2)\leq \pi$, and one of the inequalities is strict if $P$ is the triangular prism.
\item[$\mathrm{\mathbf{(a2)}}$] Whenever three distinct edges $e_1,e_2,e_3$ of $P$ meet at a vertex, then $\sum_{\mu=1}^3\Theta(e_\mu)>\pi$, and
    $\Theta(e_1)+\Theta(e_2)<\Theta(e_3)+\pi$, $\Theta(e_2)+\Theta(e_3)<\Theta(e_1)+\pi$, $\Theta(e_3)+\Theta(e_1)<\Theta(e_2)+\pi$.
\item[$\mathrm{\mathbf{(a3)}}$] Whenever $\Gamma$ is a prismatic 3-circuit intersecting edges $e_1,e_2,e_3$, then $\sum_{\mu=1}^3\Theta(e_\mu)<\pi$.
\item[$\mathrm{\mathbf{(a4)}}$] Whenever $\Gamma$ is a prismatic 4-circuit intersecting edges $e_1,e_2,e_3,e_4$, then $\sum_{\mu=1}^4\Theta(e_\mu)<2\pi$.
\end{itemize}
Then there exists a compact convex hyperbolic polyhedron $Q$ combinatorially equivalent to $P$ with dihedral angles given by $\Theta$. Furthermore, $Q$ is unique up to isometries of $\mathbb H^3$.
\end{theorem}

\begin{remark}
It is somewhat unexpected that there exists a convex subset broader than Andreev's characterization~\cite{Andreev} included in the moduli space of compact convex hyperbolic polyhedra of a given combinatorial type, in contrast with the non-convexity of the whole space indicated by D\'{\i}az~\cite{Diaz1,Diaz2} and Roeder~\cite{Roeder}.
\end{remark}

To acquire Theorem~\ref{T-1-1}, we will introduce a series of configuration spaces which  consist of circle patterns satisfying finer and finer properties. In this way we reduce the proof to showing the associated test map is surjective, which can be solved via topological degree theory. For this purpose, one needs to check the test map is proper and thus the topological degree is well-defined. Specifically, it suffices to deduce that under suitable conditions a sequence of circle patterns can produce a limit pattern with required properties.
Let $W$ denote the set of functions $\Theta:E\to (0,\pi)$ satisfying conditions $\mathrm{\mathbf{(c1)}}-\mathrm{\mathbf{(c4)}}$. Precisely, we will show the following Normal Family Theorem for circle patterns.

\begin{theorem}\label{T-1-5}
Let $\{\mathcal P_n\}$ be a sequence of \textbf{irreducible} circle patterns on $\hat{\mathbb C}$, each of which  has contact graph isomorphic to the $1$-skeleton of a triangulation $T$ of the sphere with more than four vertices.  Then $\{\mathcal P_n\}$ (modulo linear and anti-linear fractional maps) contains a subsequence convergent to an \textbf{irreducible} circle pattern with the same contact graph if the corresponding overlap angle function sequence $\{\Theta_n\}$ stays in a compact subset of $W$.
\end{theorem}

Let us give an outline of the proof of this theorem. We first endow $\hat{\mathbb C}$ with the spherical metric and assume each $\mathcal P_n$ satisfies some suitable normalization conditions. We easily see $\{\mathcal P_n\}$ contains a subsequence convergent to a pre-pattern $\mathcal P_\infty$ in certain compact extension space. It remains to verify that $\mathcal P_\infty$ satisfies the following  properties:
\begin{itemize}
\item[$\mathrm{\mathbf{\langle p1\rangle}}$] Each pre-disk in $\mathcal P_\infty$ is neither the entire sphere nor a point;
\item[$\mathrm{\mathbf{\langle p2\rangle}}$] $\mathcal P_\infty$ has contact graph isomorphic to the $1$-skeleton of $T$. Namely, no extraneous overlap is created in $\mathcal P_\infty$;
\item[$\mathrm{\mathbf{\langle p3\rangle}}$]$\mathcal P_\infty$ is  \textbf{irreducible}.
\end{itemize}

To assert $\mathrm{\mathbf{\langle p1\rangle}}$, we turn to an estimate based on the irreducibility assumption (see Lemma~\ref{L-4-5}) and an analog of the Ring Lemma of Rodin-Sullivan~\cite{Rodin-Sullivan} (see Lemma~\ref{L-4-6}). With respect to $\mathrm{\mathbf{\langle p2\rangle}}$,  as pointed out by
Bowers-Stephenson~\cite[p.\,214]{Bowers-Stephenson}, it is one of the main obstructions for extending the KAT Theorem. Fortunately, our approach reduces the problem to showing there is no "tangential type" of extraneous overlaps, which can be settled through some results  relating such extraneous overlaps to violations of conditions $\mathrm{\mathbf{(c1)}}-\mathrm{\mathbf{(c4)}}$ (see Lemma~\ref{L-4-4} and Lemma~\ref{L-4-7}). Finally, $\mathrm{\mathbf{\langle p3\rangle}}$ follows from several elementary facts (see Lemma~\ref{L-3-2} and Lemma~\ref{L-4-3}). To summarize, the entire proof can be seen in Section~\ref{S-5a}.

Notice that the above strategy is parallel to the "weak solution/regularity theory" scheme in PDE theory, where $\mathcal P_\infty$ plays a similar role to a weak solution, the lemmas mentioned above play similar roles to prior estimates, and configuration spaces play similar roles to Sobolev spaces.

The paper is organized as follows. In next section, assuming Theorem~\ref{T-1-5}, we give a quick proof of Theorem~\ref{T-1-1} by employing configuration spaces and topological degree theory. In Section~\ref{S-3a}, we derive Theorem~\ref{T-1-3} from Theorem~\ref{T-1-1}. To this end, we actually prove a theorem on which circle patterns induce compact convex hyperbolic polyhedra (see Theorem~\ref{T-3-1}). In Section~\ref{S-4a}, we establish several lemmas which serve as "prior estimates" for circle patterns. In Section~\ref{S-5a}, we prove Theorem~\ref{T-1-5}. In Section~\ref{S-6a}, we pose some questions for further developments. Particularly, we pose a conjecture which potentially provides a unified generalization of Theorem~\ref{T-1-3} and Bao-Bonahon's Theorem~\cite{Bao-Bonahon}. The last section is an appendix which complements the proofs of two lemmas in Section~\ref{S-4a}.

\bigskip
\noindent\textbf{Statement}.  We mention that Theorem~\ref{T-1-1} and Theorem~\ref{T-1-3}  grow out from an early version of the preprint~\cite{Zhou}, where the author investigated two classes of circle patterns through different geometry models. To improve readability, we  decide to divide the former manuscript into two articles. Specifically, in the latest version of the preprint~\cite{Zhou}, we restrict to circle patterns having at least one interstice. For patterns without any interstice, these are considered in the current paper.

%\medskip
%\noindent
%\textbf{Notational convention}. Throughout this paper, we use $\chi(\cdot)$ to denote the Euler characteristic of a manifold.

\section{Circle patterns from the viewpoint of manifolds}\label{S-2a}
In this section we shall give a quick proof of Theorem~\ref{T-1-1} assuming Theorem~\ref{T-1-5}. For this purpose, we introduce several configuration spaces which parameterize circle patterns with finer and finer properties. It is worth pointing out that such a viewpoint has been rooted in the works of Bauer-Stephenson-Wegert~\cite{Bauer-Stephenson-Wegert}, Zhou~\cite{Zhou1}, Bowers-Bowers-Pratt~\cite{Bowers-Bowers-Pratt}, Connelly-Gortler-Theran~\cite{Connelly-Gortler-Theran} and Connelly-Gortler~\cite{Connelly-Gortler} in consideration of computational mechanisms (see~\cite{Bauer-Stephenson-Wegert,Connelly-Gortler}) or understanding of rigidity (see~\cite{Zhou1,Bowers-Bowers-Pratt,Connelly-Gortler-Theran}).

To begin with, we endow $\hat{\mathbb C}$ with the following Riemannian metric
\[
\mathrm{d s}=\frac{2|\mathrm{d}z|}{1+|z|^2}.
\]
Note that $(\hat{\mathbb C},\mathrm{d s})$ is isometric to the unit sphere $\mathbb S^2$ in $\mathbb R^3$. In what follows we shall not distinguish  $(\hat{\mathbb C},\mathrm{d s})$ and $\mathbb S^2$ for the sake of simplicity.

Recall that $T$ is a triangulation of the sphere with more than four vertices. Let $V,E,F$ denote the sets of vertices, edges and triangles of $T$, respectively.  Set $M=\hat{\mathbb C}^{|V|}\times (0,\pi)^{|V|}$. Then $M$ is a smooth manifold parameterizing the space of patterns of $|V|$ disks on $(\hat{\mathbb C},\mathrm{d s})$. Since $T$ is a triangulation, it is trivial to see
\[
2|E|=3|F|.
\]
Combining with Euler's formula, we have
\[
\begin{aligned}
\dim (M)&=3|V|\,\\
&=|E|+(2|E|-3|F|)+3\big(|V|-|E|+|F|\big)\\
&=|E|+6.
\end{aligned}
\]

A point $(\mathbf{z},\mathbf{r})=(z_1,\cdots,z_{|V|},r_1,\cdots,r_{|V|})\in M$ is called a configuration, since it assigns each vertex $v_i\in V$ a closed disk $D(v_i)$, where $D(v_i)$ is centered at $z_i$ and is of radius $r_i$. For an edge $e=[v_i,v_j]\in E$, the inversive distance $I(e,\mathbf{z},\mathbf{r})$ is %defined to be
\[
\frac{\cos r_i\cos r_j-\cos d(z_i,z_j)}{\sin r_i\sin r_j},
\]
where $d(z_i,z_j)$ denotes the spherical distance between $z_i$ and $z_j$.

Let us introduce several subspaces of $M$. The first subspace $M_{E}\subset M$ is the set of configurations subject to
\[
-1<I(e,\mathbf{z},\mathbf{r})<1
\]
for every $e\in E$. We can then compute the overlap angle  via the formula
\[
\Theta(e,\mathbf{z},\mathbf{r})=\arccos I(e,\mathbf{z},\mathbf{r}).
\]
Next we define $M_{\Theta}\subset M_{E}$ to be the subspace of configurations with overlap angle functions satisfying conditions $\mathrm{\mathbf{(c2)}}-\mathrm{\mathbf{(c4)}}$ and the condition below:
\begin{itemize}
\item[$\mathrm{\mathbf{(x1)}}$]If $e_1,e_2$ form a simple arc with non-adjacent endpoints, then $\Theta(e_1)+\Theta(e_2)<\pi$.
\end{itemize}
In addition, let $M_{G}\subset M_{\Theta}$ consist of configurations with the following property:
\begin{itemize}
\item[$\mathrm{\mathbf{(x2)}}$] Whenever $v_\alpha, v_\beta$ are distinct, non-adjacent vertices, then $D(v_\alpha)\cap D(v_\beta)=\emptyset$.
\end{itemize}
Apparently, each $(\mathbf{z},\mathbf{r})\in M_{G}$ gives a circle pattern with contact graph isomorphic to the $1$-skeleton of $T$. Meanwhile, let $M_{IG}\subset M_{G}$ consist of configurations for which the irreducibility property holds:
\begin{itemize}
\item[$\mathrm{\mathbf{(x3)}}$] For any strict subset $A\subsetneq V$,  then $\cup_{v_i\in A} D(v_i)\subsetneq \hat{\mathbb C}$.
\end{itemize}
Note that $M_E, M_{\Theta}, M_{G}, M_{IG}$ are open subsets of $M$ and thus are smooth manifolds.

Let $v_a,v_b,v_c$ be the vertices of a triangle of $T$ and let $e_a,e_b,e_c$ be the edges opposite to $v_a,v_b,v_c$, respectively. We use $M_{IG}^\star\subset M_{IG}$ to represent the subspace of configurations such that the following normalization conditions are satisfied:
\begin{itemize}
\item[$\mathrm{\mathbf{(x4)}}$]
    $z_a=0$,\;\, $z_b>0$,\;\, $\operatorname{Im}(z_c)>0$;
\item[$\mathrm{\mathbf{(x5)}}$]
$r_a=r_b=r_c=\pi/2$.
\end{itemize}

Let $W_0$ be the set of functions $\Theta:E\to (0,\pi)$ satisfying conditions $\mathrm{\mathbf{(x1)}}$ and $\mathrm{\mathbf{(c2)}}-\mathrm{\mathbf{(c4)}}$. We define a test map as follows:
\[
\begin{aligned}
f:\quad &M^\star_{IG} &\longrightarrow \quad &\quad\quad\quad\,\;  W_0\\
&(\mathbf{z},\mathbf{r})  &\longmapsto \quad  & \big(\Theta(e_1,\mathbf{z},\mathbf{r}),\Theta(e_2,\mathbf{z},\mathbf{r}),\cdots\big).
\end{aligned}
\]

To derive Theorem~\ref{T-1-1}, we need to demonstrate  $f$ is a surjective map. For this purpose, we will make the use of topological degree theory.  Note that one needs to prove nothing if $W_0$ is an empty set. So from now on we focus on the situation that $W_0\neq\emptyset$. We first observe the following fact.

\begin{lemma}\label{L-2-1}
If $W_0\neq \emptyset$, then $W_0\cap (0,\pi/2)^{|E|}\neq \emptyset$.  Furthermore, $W_0\cap (0,\pi/2)^{|E|}$ has positive measure in $(0,\pi)^{|E|}$.
\end{lemma}
\begin{proof}
Choose $\Theta_\ast\in W_0$. For $s\in(0,1]$, set
\[
\Theta_s=s\Theta_\ast+\frac{(1-s)\pi}{3}.
\]
Taking $s_0$ sufficiently close to zero, we have $\Theta_{s_0}\in W_0\cap (0,\pi/2)^{|E|}$, which yields
\[
W_0\cap (0,\pi/2)^{|E|}\neq\emptyset.
\]
For $\delta>0$ , let $U_\delta$ consist of functions $\Theta:E\to(0,\pi)$ such that
\[
\Theta_{s_0}(e)<\Theta(e)<\Theta_{s_0}(e)+\delta,\;\;\; \forall e\in E.
\]
When $\delta$ is sufficiently small, it is easy to see $U_\delta\subset W_0\cap (0,\pi/2)^{|E|}$. Hence $W_0\cap(0,\pi/2)^{|E|}$ has positive measure in $(0,\pi)^{|E|}$.
\end{proof}

Recall that the test map $f$ is smooth. Applying Sard's Theorem~\cite{Guillemin-Pollack,Hirsch,Milnor,Lee}, one can find a regular value $\widetilde{\Theta}\in U_\delta\subset W_0\cap(0,\pi/2)^{|E|}$ of  $f$.
In view of Thurston's interpretation~\cite{Thurston}, the following lemma is a consequence of Andreev's Theorem~\cite{Andreev}. See also the work of Bowers-Stephenson~\cite{Bowers-Stephenson} for a more direct proof.

\begin{lemma}\label{L-2-2}
There exists an \textbf{irreducible} circle pattern $\widetilde{\mathcal P}$ on the Riemann sphere  $\hat{\mathbb C}$ with contact graph isomorphic to the $1$-skeleton of $T$ and overlap angles given by
$\widetilde{\Theta}$. Moreover, $\widetilde{\mathcal P}$ is unique up to linear and anti-linear fractional maps.
\end{lemma}

Notice that $\widetilde{\mathcal P}$ can be transformed into a normalized pattern through a unique linear or anti-linear fractional map, which indicates the following lemma.
\begin{lemma}
If $W_0\neq\emptyset$, then $M^\star_{IG}\neq \emptyset$ is a smooth manifold of dimension $|E|$.
\end{lemma}

Since $W_0$ is an open subset of $(0,\pi)^{|E|}$, it follows that
\[
\dim(M^\star_{IG})=\dim (W_0)=|E|.
\]
Meanwhile, assuming Theorem~\ref{T-1-5}, we assert the properness of the test map.
\begin{lemma}
Under the hypothesis that $W_0\neq\emptyset$, the map $f:M^\star_{IG}\to W_0$ is proper.
\end{lemma}
\begin{proof}
For any compact subset $\Omega$ of $W_0$, we need to check  $f^{-1}(\Omega)$ is compact. It suffices to show any sequence $\{(\mathbf{z}_n,\mathbf{r}_n)\}$ in $f^{-1}(\Omega)$ contains a convergent subsequence. Note that each $(\mathbf{z}_n,\mathbf{r}_n)$ gives a normalized \textbf{irreducible} circle patterns $\mathcal P_n$ on $\hat{\mathbb C}$ with contact graph isomorphic to the $1$-skeleton of $T$ and overlap angle function staying in $\Omega$. Applying Theorem~\ref{T-1-5}, we immediately conclude the lemma.
\end{proof}

By manifold theory (see, e.g.,~\cite{Guillemin-Pollack,Hirsch,Milnor,Lee}), the topological degree of $f$ is well-defined. Furthermore, we have the following property.
\begin{proposition}\label{P-2-5}
Suppose $W_0\neq\emptyset$. Then $\deg(f)=\pm1$.
\end{proposition}
\begin{proof}
Because $W_0$ is connected, we can choose an arbitrary regular value of $f$ to calculate the degree. In particular, we derive
\[
\deg(f)=\sum_{(\mathbf{z},\mathbf{r})\in f^{-1}(\widetilde{\Theta})} \operatorname{sgn}\left(\operatorname{det}(Df(\mathbf{z},\mathbf{r}))\right),
\]
where $Df(\mathbf{z},\mathbf{r})$ is the Jacobian matrix of $f$ at $(\mathbf{z},\mathbf{r})$. Meanwhile, it follows from Lemma~\ref{L-2-2} that $f^{-1}(\widetilde{\Theta})$ consists of a single point, which yields $\deg(f)=\pm1$.
\end{proof}

According to a basic property of topological degree (see, e.g.,~\cite{Guillemin-Pollack,Hirsch,Milnor,Lee}), the proposition below is  straightforward.
\begin{proposition}\label{P-2-6}
The map $f:M^\star_{IG}\to W_0$ is surjective.
\end{proposition}

It is timely to show the main theorem.
\begin{proof}[\textbf{Proof of Theorem~\ref{T-1-1} (assuming Theorem~\ref{T-1-5})}]
For each positive integer $n$, set
\[
\Theta_n=\frac{n\Theta}{n+1}+\frac{\pi}{3(n+1)}.
\]
Clearly, $\Theta_n\in W_0\subset W$. Owing to Proposition~\ref{P-2-6}, there exists an \textbf{irreducible} circle pattern $\mathcal P_n$ on $\hat{\mathbb C}$ realizing the data $(T, \Theta_n)$.  By Theorem~\ref{T-1-5}, we can pick up a subsequence of $\{\mathcal P_n\}$ convergent to an \textbf{irreducible} circle pattern $\mathcal P$ with contact graph isomorphic to the $1$-skeleton of $T$. Moreover, it is easy to see $\mathcal P$ has overlap angles given by $\Theta$, which asserts the existence part of the theorem. Finally, the rigidity part follows from the rigidity of the corresponding hyperbolic polyhedra, which will be asserted in Theorem~\ref{T-1-3}.
\end{proof}

\section{Generalizing Andreev's Theorem}\label{S-3a}
We first recall the correspondence between circle patterns and hyperbolic polyhedra observed by Thurston~\cite[Chap. 13]{Thurston}. To be specific, given a circle pattern $\mathcal P=\{D(v)\}_{v\in V}$ on $\hat{\mathbb C}=\partial\mathbb H^3$ with contact graph isomorphic to the $1$-skeleton of a triangulation $T$, we associate each disk  $D(v)$ with the half-space $H(v)\subset\mathbb H^3$ identical to the closed convex hull of ideal points in $\hat{\mathbb C}\setminus D(v)\subset\partial \mathbb H^3$. This gives rise to a map sending each pattern $\mathcal P$ to a convex set $Q(\mathcal P)\subset \mathbb H^3$, where $Q({\mathcal P})=\cap_{v\in V}H(v)$. Moreover, if $Q(\mathcal P)$ is a compact polyhedron, then it has boundary decomposition dual to $T$ and dihedral angles equal to the corresponding overlap angles of $\mathcal P$.

Unfortunately, $Q(\mathcal P)$ is possibly not a compact polyhedron. A natural problem asks which circle patterns induce compact polyhedra. Here we give a partial answer.

\begin{theorem}\label{T-3-1}
Let $T$ be a triangulation of the sphere with more than four vertices and let $\mathcal P=\{D(v)\}_{v\in V}$ be a circle pattern on $\hat{\mathbb C}$ with contact graph isomorphic to the $1$-skeleton of $T$ and overlap angle function satisfying condition $\mathrm{\mathbf{(c2)}}$. Assume the following condition holds:
\begin{itemize}
\item[$\mathrm{\mathbf{(cr)}}$]Whenever the edges $[v_\alpha,v_\beta],[v_\beta,v_\gamma],[v_\gamma,v_\alpha]$ form a simple closed curve separating the vertices of $T$, then $D(v_\alpha)\cap D(v_\beta)\cap D(v_\gamma)=\emptyset$.
\end{itemize}
Then $Q(\mathcal P)$ is a compact convex hyperbolic polyhedron with boundary decomposition  dual to $T$ and dihedral angles equal to the corresponding overlap angles of $\mathcal P$.
\end{theorem}

\begin{proof}
For each $v\in V$, set $\Pi(v)=\partial_{\mathbb H^3} H(v)$. Let $v_i,v_j,v_k$ be the vertices of a triangle of $T$ and let $e_i,e_j,e_k$ be the edges opposite to $v_i,v_j,v_k$, respectively. We first claim $\Pi(v_i),\Pi(v_j),\Pi(v_k)$ intersect at a point in $\mathbb H^3$, which is reduced to checking the Gram matrix $J(v_i,v_j,v_k)$ of the geodesic planes $\Pi(v_i),\Pi(v_j),\Pi(v_k)$ is positive-definite. Note that $J(v_i,v_j,v_k)$ is equal to
\[
\left[\begin{array}{ccc}
1 & -\cos \theta_k & -\cos \theta_j \\
-\cos \theta_k & 1 & -\cos \theta_i \\
-\cos \theta_j & -\cos \theta_i & 1
\end{array}\right],
\]
where $\theta_l=\Theta(e_l)$ for $l=i,j,k$. Owing to Sylvester's criterion, we need to prove the $m$-th leading principal minor of $J(v_i,v_j,v_k)$ is positive for $m=1,2,3$. The cases that $m=1,2$ are straightforward. For $m=3$, it is equivalent to showing the determinant of $J(v_i,v_j,v_k)$ is also positive. Indeed, a routine calculation indicates the determinant is equal to
\[
-4 \cos\dfrac{\theta_i+\theta_j+\theta_k}{2}
\cos\dfrac{\theta_i+\theta_j-\theta_k}{2}
\cos\dfrac{\theta_j+\theta_k-\theta_i}{2}
\cos\dfrac{\theta_k+\theta_i-\theta_j}{2},
\]
Under condition $\mathrm{\mathbf{(c2)}}$, the above formula is positive, which asserts the claim.

Meanwhile, by Steinitz's Theorem, we have an abstract compact polyhedron $Q(T)$ with boundary $\partial Q(T)$ being a cellular decomposition dual to $T$. We now reduce the theorem to showing there exists a homeomorphism between $\partial Q(T)$ and $\partial_{\mathbb H^3} Q(\mathcal P)$. For closed subsets $A, B$ of certain topological space, recall Leibniz's rule:
\[
\partial(A\cap B)=(\partial A\cap B)\cup (A\cap\partial B).
\]
It follows that
\[
\partial_{\mathbb H^3} Q(\mathcal P)\,=\,\partial_{\mathbb H^3} \left(\cap_{v\in V}H(v)\right)\,=\,\cup_{v\in V} \mathcal F(v),
\]
where
\begin{equation}\label{E-3-1}
\mathcal F(v)=\big( \cap_{u\in V, u\neq v}H(u)\big)\cap\partial_{\mathbb H^3}H(v)=\big( \cap_{u\in V, u\neq v}H(u)\big)\cap\Pi(v).
\end{equation}
For each $u\in V\setminus\{v\}$ which is not adjacent to $v$, we claim
\begin{equation}\label{E-3-2}
\Pi(v)\subset H(u).
\end{equation}
Otherwise, $\partial D(v)\cap D(u)\neq \emptyset$, which contradicts that $D(v),D(u)$ are disjoint. Combining~\eqref{E-3-1} and~\eqref{E-3-2}, we derive
\begin{equation}\label{E-3-3}
\mathcal F(v)=\big(\cap_{\mu=1}^d H(u_\mu)\big)\cap \Pi(v),
\end{equation}
where $u_1,\cdots,u_d$, in anticlockwise order, are all adjacent vertices of $v$.

Let $\sum(v)$ be the face of $\partial Q(T)$ dual to $v$. Then $\sum(v)$ is a topological polygon marked by triangles $f_1,\cdots,f_d$ of $T$, where $f_\tau$ has vertices $v, u_\tau, u_{\tau+1}$ for $\tau=1,\cdots,d$ (set $u_{d+1}=u_1$). We will prove there exists a mark-preserving homeomorphism between $\sum(v)$ and $\mathcal F(v)$. To this end, let $\partial_{\Pi} \mathcal F(v)$ denote the boundary of $\mathcal F(v)$ in the geodesic plane $\Pi(v)$. Applying Leibniz's rule to~\eqref{E-3-3}, we obtain
\begin{equation}\label{E-3-4}
\partial_{\Pi} \mathcal F(v)=\cup_{\mu=1}^d \left(\big(\,\cap_{\eta=1,\eta\neq\mu}^dH(u_\eta)\,\big)\cap \Pi(u_\mu)\cap \Pi(v)\right).
\end{equation}
When $\eta\neq \mu-1,\mu,\mu+1$, note that
\[
\Pi(u_\mu)\cap \Pi(v)\subset H(u_\eta).
\]
Otherwise, $\partial D(u_\mu)\cap\partial D(v)\cap D(u_\eta)\neq \emptyset$, which contradicts~\eqref{E-3-5} in the following Lemma~\ref{L-3-2}. We thus simplify formula~\eqref{E-3-4} as
\[
\partial_{\Pi}\mathcal F(v)=\cup_{\mu=1}^d\left( H(u_{\mu-1})\cap H(u_{\mu+1})\cap \Pi(u_\mu)\cap \Pi(v)\right):=\cup_{\mu=1}^d \mathcal L_{\mu}.
\]
It is easy to see $\mathcal L_{\mu}$ is a geodesic segment with endpoints $q(f_{\mu-1}),q(f_{\mu})$, where
 \[
 q(f_\tau)=\Pi(u_\tau)\cap\Pi(u_{\tau+1})\cap\Pi(v)
 \]
for $\tau=1,\cdots,d$. That means $\mathcal F(v)$ is a convex polygon with vertices $q(f_1),\cdots,q(f_d)$. Hence the map
\[
(f_1,\cdots,f_d)\longmapsto(q(f_1),\cdots,q(f_d))
\]
induces a mark-preserving homeomorphism between $\sum(v)$ and $\mathcal F(v)$.

Using the Pasting Lemma, we thus obtain a global homeomorphism between $\partial Q(T)$ and $\partial_{\mathbb H^3} Q(\mathcal P)$, which implies $Q(\mathcal P)$ is a compact polyhedron with boundary decomposition dual to $T$. Finally, the property that $Q(\mathcal P)$ has dihedral angles equal to the corresponding overlap angles of $\mathcal P$ is straightforward.
\end{proof}

The following lemma, required in the above proof,  unveils some properties behind  conditions of Theorem~\ref{T-3-1}.
\begin{lemma}\label{L-3-2}
Let $T,\mathcal P$ be as in Theorem~\ref{T-3-1}. If $v_i,v_j,v_k$ are distinct vertices of $T$ such that
\begin{equation}\label{E-3-5}
D(v_i)\cap D(v_j)\cap D(v_k)\neq\emptyset,
\end{equation}
then $v_i,v_j,v_k$ are the vertices of a triangle of $T$. As a result,
\begin{equation}\label{E-3-6}
D(v_i)\cap D(v_j)\cap D(v_k)\cap D(v_l)=\emptyset
\end{equation}
for any four distinct vertices $v_i,v_j,v_k,v_l$ of $T$.
\end{lemma}

\begin{proof}
First we claim any two distinct vertices among $v_i,v_j,v_k$ are adjacent. Otherwise, without loss of generality, suppose $v_i$ is not adjacent to $v_j$. Then $D(v_i)\cap D(v_j)=\emptyset$, which contradicts~\eqref{E-3-5}. We further claim $v_i,v_j,v_k$ are the vertices of a triangle of $T$. Otherwise, the edges $[v_i,v_j],[v_j ,v_k], [v_k,v_i]$ form a simple closed curve separating the vertices of $T$. By condition $\mathrm{\mathbf{(cr)}}$, this also leads a contradiction to~\eqref{E-3-5}. Finally, remembering $T$ is not the boundary of a tetrahedron, we immediately derive~\eqref{E-3-6} from~\eqref{E-3-5}.
\end{proof}

Before the proof of Theorem~\ref{T-1-3}, we recall a simple fact which relates the irreducibility property to condition $\mathrm{\mathbf{(cr)}}$.
\begin{lemma}\label{L-3-3}
Let $D_i, D_j, D_k$ be three mutually intersecting disks on
$\hat{\mathbb C}$ with overlap angles $\Theta_{ij},\Theta_{jk},\Theta_{ki}\in[0,\pi)$. Suppose
$\mathbb{D}_i\cup \mathbb{D}_j\cup \mathbb{D}_k\subsetneq\hat{\mathbb C}$, where $\mathbb D_\mu$ denotes the interior of
$D_\mu$ for $\mu=i,j,k$. If $\Theta_{ij}+\Theta_{jk}+\Theta_{ki}<\pi$, then
\[
D_i\cap D_j\cap D_k=\emptyset.
\]
\end{lemma}

We are led to the general Andreev's Theorem.
\begin{proof}[\textbf{Proof of Theorem~\ref{T-1-3}}]
Let $T$ be the triangulation dual to the boundary decomposition of $P$ and let $\Theta^\ast: E\to(0,\pi)$ be the function assigning each edge of $T$ with the angle given by $\Theta$ at the dual edge. Since $\Theta:\mathcal E\to(0,\pi)$ satisfies conditions $\mathrm{\mathbf{(a1)}}-\mathrm{\mathbf{(a4)}}$, it is easy to see $\Theta^\ast: E\to(0,\pi)$ satisfies conditions
$\mathrm{\mathbf{(c1)}}-\mathrm{\mathbf{(c4)}}$. By Theorem~\ref{T-1-1}, there exists an \textbf{irreducible} circle pattern $\mathcal P$ on $\hat{\mathbb C}$ realizing $(T,\Theta^\ast)$.

Suppose $v_\alpha,v_\beta,v_\gamma$ are three vertices such that $[v_\alpha,v_\beta],[v_\beta ,v_\gamma], [v_\gamma,v_\alpha]$ form a simple closed curve separating the vertices of $T$. Because $\mathcal P$ is \textbf{irreducible} and $T$ possesses more than four vertices, we have $D(v_\alpha)\cup D(v_\beta)\cup D(v_\gamma)\subsetneq \hat{\mathbb C}$. Recalling $\Theta^\ast: E\to(0,\pi)$ satisfies conditions
$\mathrm{\mathbf{(c3)}}$,  we then check condition $\mathrm{\mathbf{(cr)}}$ by Lemma~\ref{L-3-3}. The existence part of the theorem thus follows from Theorem~\ref{T-3-1}. Finally, the rigidity part is a straightforward consequence of the following theorem due to Rivin-Hodgson~\cite[Corollary 4.6]{Rivin-Hodgson}.
\end{proof}

\begin{theorem}[Rivin-Hodgson]\label{T-3-4}
Compact convex polyhedra in $\mathbb H^3$ with trivalent vertices are determined up to isometries by their combinatorics and dihedral angles.
\end{theorem}

\begin{remark}
Rivin-Hodgson~\cite{Rivin-Hodgson} proved the above theorem via analogous reasoning to Cauchy's Rigidity Theorem (see~\cite[Chap.\,12]{Aigner-Ziegler}). Remind similar results can be seen in the works of Andreev~\cite[Theorem 4]{Andreev} and Roeder-Hubbard-Dunbar~\cite[Proposition 4.1]{Roeder-Hubbard-Dunbar}. On the other hand, the rigidity of compact convex hyperbolic polyhedra of more general combinatorial types with respect to their dihedral angles is an open problem dated back to Stoker~\cite{Stoker} (independently posed by Schlenker~\cite[Question 1]{Schlenker1}), which motivated the Infinitesimal Rigidity Theorem of Mazzeo-Montcouquiol~\cite{Mazzeo-Montcouquiol}. See also the fruitful research of Hodgson-Kerckhoff~\cite{Hodgson-Kerckhoff}, Weiss~\cite{Weiss} and Luo-Yang~\cite{Luo-Yang}.
\end{remark}

\section{Prior estimates for circle patterns}\label{S-4a}
As mentioned before, we regard configurations spaces as discrete analogs of Sobolev spaces. To prove Theorem~\ref{T-1-5}, it is necessary to establish some prior estimates relating different configurations spaces. For quickly detecting the ideas, the readers might focus on the core information of these estimates at first and revisit the proofs in details after getting the big picture of the paper.

We begin with several lemmas which serve as local estimates. Mention that some relevant results have already appeared in the works of Thurston~\cite{Thurston},
Marden-Rodin~\cite{Marden-Rodin}, Bowers-Stephenson~\cite{Bowers-Stephenson}, Chow-Luo~\cite{Chow-Luo} and others.

\begin{lemma}\label{L-4-1}
Suppose $\Theta_{ij},\Theta_{jk},\Theta_{ki}\in(0,\pi)$ are three angles satisfying
\[
\Theta_{ij}+\Theta_{jk}+\Theta_{ki}>\pi,\;\;
\Theta_{ij}+\Theta_{jk}<\Theta_{ki}+\pi,\;\;
\Theta_{jk}+\Theta_{ki}<\Theta_{ij}+\pi,\;\;
\Theta_{ki}+\Theta_{ij}<\Theta_{jk}+\pi.
\]
For any three reals $r_i,r_j,r_k\in(0,\pi)$,  there exists a configuration of three intersecting disks $D_i,D_j,D_k$ on $(\hat{\mathbb{C}},\mathrm{ds})$, unique up to isometries, having radii $r_i, r_j, r_k$  and meeting in overlap angles $\Theta_{ij},\Theta_{jk},\Theta_{ki}$. Moreover, the centers of these disks determine a spherical triangle.
\end{lemma}

\begin{figure}[htbp]\centering
\includegraphics[width=0.52\textwidth]{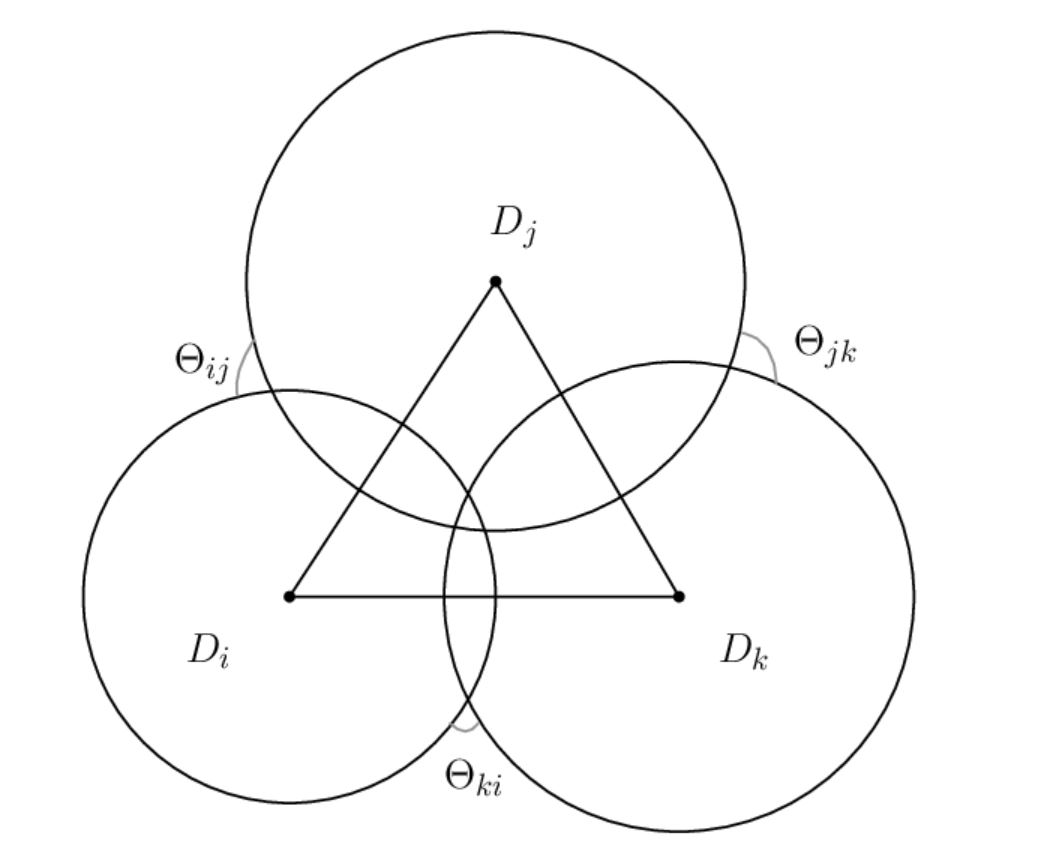}
\caption{Three-circle configuration}%\label{F-3}
\end{figure}

The proof is analogous to the hyperbolic and Euclidean versions established by Zhou~\cite{Zhou1} and Jiang-Luo-Zhou~\cite{Jiang-Luo-Zhou},  but the details are complicated. We postpone it to the appendix. Meanwhile, we observe the following simple fact.

\begin{lemma}\label{L-4-2}
Let $D_i,D_j,D_k$ be as in Lemma~\ref{L-4-1}. Then $\partial D_i\cap \partial D_j\cap \partial D_k=\emptyset$.
\end{lemma}

\begin{lemma}\label{L-4-3}
Let $D_i, D_j, D_k$ be three intersecting disks on $(\hat{\mathbb C},\mathrm{ds})$ with radii $r_i,r_j,r_k\in(0,\pi/2]$ and overlap angles $\Theta_{ij},\Theta_{jk},\Theta_{ki}\in[0,\pi)$. If $\Theta_{ij}+\Theta_{jk}+\Theta_{ki}<\pi$, then
\[
D_i\cap D_j\cap D_k=\emptyset.
\]
\end{lemma}
\begin{proof}
Let $\mathbb D_i, \mathbb D_j, \mathbb D_k$ be the interiors of $D_i, D_j, D_k$, respectively. Since $r_i,r_j\in(0,\pi/2]$, it is easy to see the set $\hat{\mathbb C}\setminus(\mathbb D_i\cup\mathbb D_j)$ contains a pair of antipodal points. Because $r_k\leq\pi/2$, one of the antipodal points can not be covered by $\mathbb D_k$. That means $\mathbb D_i\cup\mathbb D_j\cup \mathbb D_k\subsetneq \hat{\mathbb C}$. Applying Lemma~\ref{L-3-3}, we  obtain $D_i\cap D_j\cap D_k=\emptyset$.
\end{proof}

\begin{lemma}\label{L-4-4}
Let $D_i, D_j, D_k$ be three mutually intersecting disks on $\hat{\mathbb C}$ with overlap angles $\Theta_{ij},\Theta_{jk},\Theta_{ki}\in[0,\pi)$. Suppose $D_i\cap D_j\subset D_k$. Then
\[
\Theta_{jk}+\Theta_{ki}  \geq \Theta_{ij} + \pi.
\]
In particular, when $D_i,D_j$ intersect at only one point, we have
\[
\Theta_{jk} +\Theta_{ki}\geq \pi,
\]
where the equality holds if and only if $D_i\cap D_j=\partial D_i\cap \partial D_j\subset\partial D_k$.
\end{lemma}

\begin{proof}
The lemma is straightforward after transforming  $\partial D_i, \partial D_j$ into intersecting lines through a suitable linear fractional map. See Figure~\ref{F-2}.
\end{proof}

\begin{comment}
\begin{figure}[htbp]
\centering
\begin{tikzpicture}[line width=0.45pt,font=\small]
\draw (-7.6,-0.3) circle (2.17);
\draw (-8.5,-1)  node [right] {$D_i$};
\draw (-5.7,-0.22) circle (1.84);
\draw (-5.1,-0.3)  node [right] {$D_k$};
\draw (-3.6,-0.22) circle (2.39);
\draw (-2.2,-0.5)  node [right] {$D_j$};
\draw[->] (-1,0.5) to [bend left](0.8,0.5);
\draw (-6.19,1.56) [line width=0.3pt] arc (30:132:0.25);
\draw (-5.06,1.69) [line width=0.3pt] arc (60:148:0.25);
\draw (-5.83,-1.56) [line width=0.3pt] arc (230:300:0.25);
\draw (-6.39,-2.11) [line width=0.3pt] arc (230:330:0.25);
\draw (-5.38,-2.03) [line width=0.3pt] arc (210:300:0.25);
\draw (3.04,-0.62) circle (2.2);
\draw (0.2,-2.5)  node [right] {$D_i$};
\draw (5.0,-2.45)  node [right] {$D_j$};
\draw (5.39,-0.3)  node [right] {$D_k$};
\path [draw] (3.1,4.38)--(0.72,-3.2);
\path [draw] (2.42,4.42)--(4.68,-3.18);
\draw (2.6,2.8) [line width=0.3pt] arc (230:308:0.25);
\draw (2.78,2.62) node [below] {$\Theta_{ij}$};
\draw (2.09,1.17) [line width=0.3pt] arc (250:376:0.25);
\draw (2.0,0.85) node [right] {$\Theta_{ki}$};
\draw (3.06,1.58) [line width=0.35pt] arc (165:282:0.25);
\draw (2.75,1.0) node [right] {$\Theta_{jk}$};
\draw (1.32,-1.97) [line width=0.3pt] arc (310:430:0.25);
\draw (4.36,-2.13) [line width=0.3pt] arc (110:208:0.25);
%
\end{tikzpicture}
\caption{Angle relation}\label{F-2}
\end{figure}
%
\end{comment}

\begin{figure}[htbp]\centering
\includegraphics[width=0.95\textwidth]{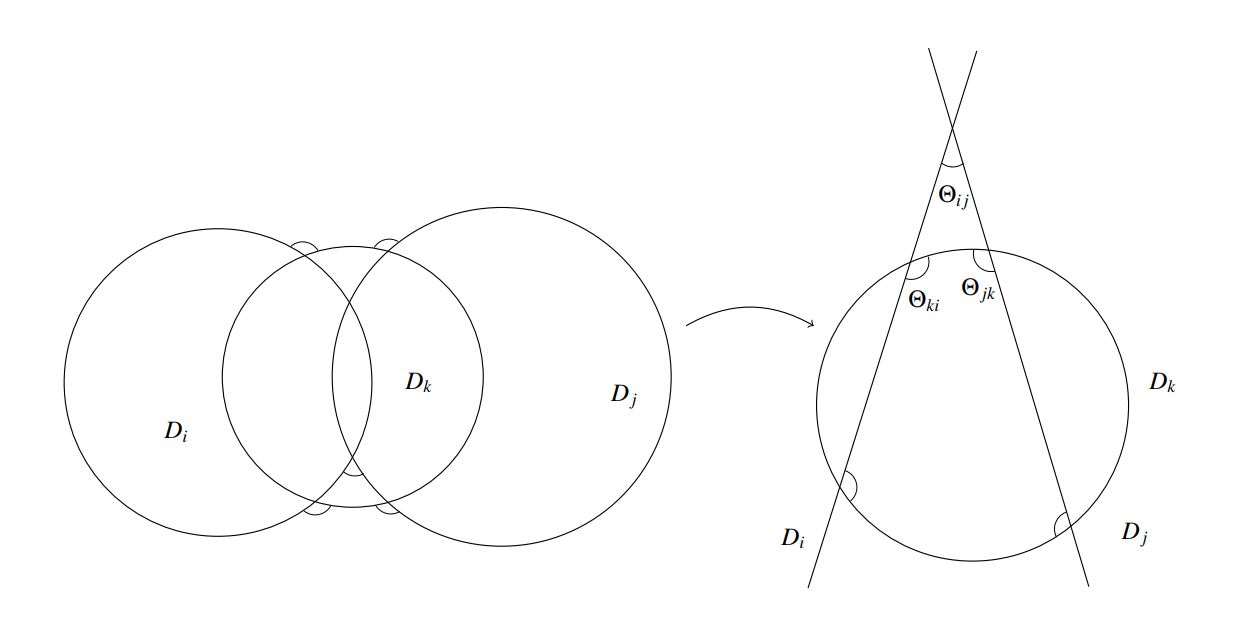}
\caption{Angle relation}\label{F-2}
\end{figure}

Next we introduce some global estimates. Unless otherwise stated, in what follows we always assume $T$ is a triangulation of the sphere with more than four vertices. For each disk $D(v)$ on $(\hat{\mathbb C},\mathrm {ds})$,  we use $r(D(v))$ to denote the radius of $D(v)$.

\begin{lemma}\label{L-4-5}
Let $\mathcal P=\{D(v)\}_{v\in V}$ be an \textbf{irreducible} circle pattern on
$(\hat{\mathbb C},\mathrm {ds})$ with contact graph isomorphic to the $1$-skeleton of $T$ and overlap angle function $\Theta:E\to(0,\pi)$ satisfying condition $\mathrm{\mathbf{(c3)}}$. Suppose $v_a,v_b,v_c$ are the vertices of a triangle of $T$. If
\[
r(D(v_a))=r(D(v_b))=r(D(v_c))=\pi/2,
\]
then
\[
r(D(v))<\pi/2,\quad \forall\,v\in V\setminus\{v_a,v_b,v_c\}.
\]
\end{lemma}
\begin{proof}
Assume on the contrary that there exists $v_0\in V\setminus\{v_a,v_b,v_c\}$ such that $
r(D(v_0))\geq \pi/2$. Then $D(v_0)\cap D(v_\mu)\neq\emptyset$ for $\mu=a,b,c$. That means $v_0$ is a common adjacent vertex of $v_a,v_b,v_c$. Since $T$ has more than four vertices, there is no loss of generality in assuming the edges $[v_a,v_b]$, $[v_b,v_0]$, $[v_0,v_b]$ form a simple closed curve separating the vertices of $T$. Because $\mathcal P$ is \textbf{irreducible} and $\Theta$ satisfies condition $\mathrm{\mathbf{(c3)}}$, Lemma~\ref{L-3-3} indicates
\[
D(v_a)\cap D(v_b)\cap D(v_0)=\emptyset.
\]
It follows that
\[
D(v_a)\cap D(v_b)\subset \hat{\mathbb C}\setminus D(v_0).
\]
On the other hand, under condition $r(D(v_a))=r(D(v_b))=\pi/2$, the intersection $D(v_a)\cap D(v_b)$ contains a pair of antipodal points. Therefore, the open disk $\hat{\mathbb C}\setminus D(v_0)$ must have radius larger than $\pi/2$, which contradicts the assumption that $r(D(v_0))\geq \pi/2$.
\end{proof}

The following is an analog of the Ring Lemma of Rodin-Sullivan~\cite{Rodin-Sullivan}. Remind  the Euclidean version in the case that $0<\Theta\leq\pi/2$ was obtained by He~\cite[Lemma 7.1]{He}.
\begin{lemma}\label{L-4-6}
Let $\mathcal P=\{D(v)\}_{v\in V}$ be an \textbf{irreducible} circle pattern on
$(\hat{\mathbb C},\mathrm {ds})$ with contact graph isomorphic to the $1$-skeleton of $T$ and overlap angle function $\Theta: E\to(0,\pi)$ staying in a compact subset $\Omega\subset W$. Suppose $r(D(v))\leq \pi/2$ for all $v\in V$. Then there exists a positive constant $C$ depending only on $T$ and $\Omega$ such that
\[
r(D(v))\leq C r(D(u))
\]
for each pair of adjacent vertices $v,u\in V$.
\end{lemma}
The lemma relies on subtle analysis with the proof postponed to the appendix.

A circle pattern $\mathcal P=\{D(v)\}_{v\in V}$ on $(\hat{\mathbb C},\mathrm{d s})$ is said to be \textbf{$T$-type} if there is a geodesic triangulation of $(\hat{\mathbb C},\mathrm{d s})$ isotopic to $T$ with vertices given by the centers of the disks in $\mathcal P$. We  write the interior of each closed disk $D(v)$ as $\mathbb D(v)$.
\begin{lemma}\label{L-4-7}
Let $\mathcal P=\{D(v)\}_{v\in V}$ be a \textbf{$T$-type} circle pattern on $(\hat{\mathbb C},\mathrm{d s})$ with overlap angle function $\Theta: E\to(0,\pi)$ satisfying condition $\mathrm{\mathbf{(c2)}}$. Suppose $v_\alpha,v_\beta\in V$ are two distinct, non-adjacent vertices. For each $p\in D(v_\alpha)\cap D(v_\beta)$,  then there exists  $v_\eta\in V\setminus\{v_\alpha,v_\beta\}$ such that $p\in\mathbb D(v_\eta)$.
\end{lemma}

\begin{proof}
As in Figure~\ref{F-3}, suppose $u_{1},\cdots, u_{d}$, in anticlockwise order, are all adjacent vertices of $v_\alpha$. For $\mu=1,\cdots,d$, let $l(v_\alpha u_\mu)$ be the geodesic ray starting from $v_\alpha$ along the direction to $u_\mu$ and let $\operatorname{Sec}(v_\alpha u_\mu u_{\mu+1})\subset D(v_\alpha)$ be the circular sector between $l(v_\alpha u_\mu)$ and $l(v_\alpha u_{\mu+1})$ (set $u_{d+1}=u_1$). Under condition $\mathrm{\mathbf{(c2)}}$, we are ready to see
\[
\operatorname{Sec}(v_\alpha u_\mu u_{\mu+1})\subset \Delta(v_\alpha u_\mu u_{\mu+1})\cup \mathbb D(u_\mu)\cup \mathbb {D}(u_{\mu+1}).
\]
where $\Delta(v_\alpha u_\mu u_{\mu+1})$ denotes the geodesic triangle with vertices $v_\alpha, u_\mu, u_{\mu+1}$. Hence
\[
D(v_\alpha)=\cup_{\mu=1}^d\operatorname{Sec}(v_\alpha u_\mu u_{\mu+1})\subset \big(\cup_{\mu=1}^d \mathbb D(u_\mu)\big)\cup\big(\cup_{\mu=1}^d\Delta(v_\alpha u_\mu u_{\mu+1})\big).
\]
Set $\operatorname{K}(v_\alpha)=\cup_{\mu=1}^d\Delta(v_\alpha u_\mu u_{\mu+1})$ and $\operatorname{\mathbb{K}}(v_\alpha)=\operatorname{K}(v_\alpha)\setminus \partial \operatorname{K}(v_\alpha)$. Notice that
\[
\partial \operatorname{K}(v_\alpha)\subset \cup_{\mu=1}^d \mathbb D(u_\mu).
\]
Consequently,
\begin{equation}\label{E-4-1}
D(v_\alpha)\subset \big(\cup_{\mu=1}^d \mathbb D(u_\mu)\big)\cup \operatorname{\mathbb{K}}(v_\alpha).
\end{equation}

\begin{figure}[htbp]\centering
\includegraphics[width=0.6\textwidth]{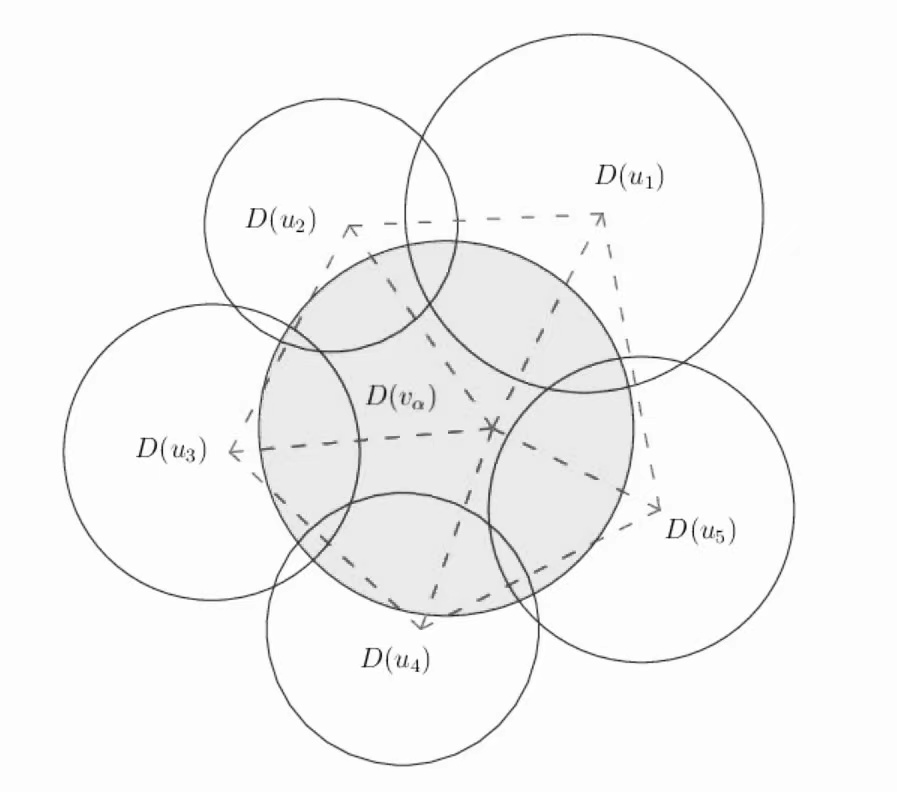}
\caption{The central disk is covered by the union of linked triangles and adjacent disks}\label{F-3}
\end{figure}

We now assume, by contradiction, that there exists
$p_0\in D(v_\alpha)\cap D(v_\beta)$ which does not belong to any open disks of $\mathcal P$ except for $\mathbb D(v_\alpha), \mathbb D(v_\beta)$. Particularly,
\[
p_0\notin \cup_{\mu=1}^d \mathbb D(u_{\mu}).
\]
Taking~\eqref{E-4-1} into consideration, one obtains
\[
p_0 \in \operatorname{\mathbb{K}}(v_\alpha).
\]
Similarly, we define $\operatorname{\mathbb{K}}(v_\beta)$ and check that
\[
p_0 \in \operatorname{\mathbb{K}}(v_\beta).
\]
Thus $\operatorname{\mathbb{K}}(v_\alpha)\cap\operatorname{\mathbb{K}}(v_\beta)$ is non-empty, which occurs if and only if there exists an edge between $v_\alpha$ and $v_\beta$. This contradicts that $v_\alpha,v_\beta$ are non-adjacent vertices.
\end{proof}

\section{Normal Family Theorem for circle patterns}\label{S-5a}
We are now in a position to prove the Normal Family Theorem. The idea is analogous to the regularization process used in PDE theory.

\begin{proof}[\textbf{Proof of Theorem~\ref{T-1-5}}]
Let $v_a,v_b,v_c$ be the vertices of a triangle of $T$. Using suitable linear  fractional maps, we may assume each pattern $\mathcal P_{n}=\{D_n(v)\}_{v\in V}$ satisfies
\begin{equation}\label{E-5-1}
r(D_n(v_a))=r(D_n(v_b))=r(D_n(v_c))=\pi/2.
\end{equation}
According to Lemma~\ref{L-4-5}, we derive
\[
r(D_n(v))\leq\pi/2,\quad \forall\, v\in V.
\]
Meanwhile, we claim there exists $\delta\in(0,\pi/2]$ such that
\[
r(D_n(v))\geq \delta,\quad \forall\, v\in V.
\]
Otherwise, by passing to a subsequence if necessary, there exists $v_0\in V$ such that
\[
r(D_n(v_0))\to 0.
\]
From Lemma~\ref{L-4-6}, it follows that
\[
r(D_n(v))\to 0,\quad \forall\, v\in V,
\]
which contradicts~\eqref{E-5-1}.

Therefore, $\{\mathcal P_n\}$ is parameterized by a sequence in
$\hat{\mathbb C}^{|V|}\times [\delta,\pi/2]^{|V|}$. Taking a subsequence if necessary, we may assume $\{\mathcal P_n\}$ converges to a circle pattern $\mathcal P_\infty=\{D_\infty(v)\}_{v\in V}$ on $(\hat{\mathbb C},\mathrm{d s})$ with the following property:
\[
\delta\leq r(D_\infty(v))\leq\pi/2,\quad \forall\, v\in V.
\]
Recalling that $\{\Theta_n\}$ stays in a compact subset of $W$, we check the overlap angle function $\Theta_\infty$ of $\mathcal P_\infty$ satisfies conditions $\mathrm{\mathbf{(c1)}}-\mathrm{\mathbf{(c4)}}$. By Lemma~\ref{L-4-1}, the geodesic triangle of centers for each face of $T$ is well-defined, which indicates $\mathcal P_\infty$ is a \textbf{$T$-type} circle pattern.

To check $\mathcal P_\infty$ has the same contact graph with $\mathcal P_n$, it suffices to show $D_{\infty}(v_\alpha),D_{\infty}(v_\beta)$ are disjoint for each pair of non-adjacent vertices $v_\alpha,v_\beta\in V$. Assume on the contrary  $D_{\infty}(v_\alpha)\cap D_{\infty}(v_\beta)\neq\emptyset$. Then it is easy to see the disks $D_{\infty}(v_\alpha), D_{\infty}(v_\beta)$ are externally tangent. Otherwise, $D_{\infty}(v_\alpha)\cap D_{\infty}(v_\beta)$ contains interior points, which implies $D_{n}(v_\alpha)\cap D_{n}(v_\beta)$ also contains interior points for $n$ sufficiently large and thus contradicts the assumption that $D_{n}(v_\alpha),D_{n}(v_\beta)$ are disjoint. Now there are two cases to distinguish:

\begin{figure}[htbp]\centering
\includegraphics[width=0.65\textwidth]{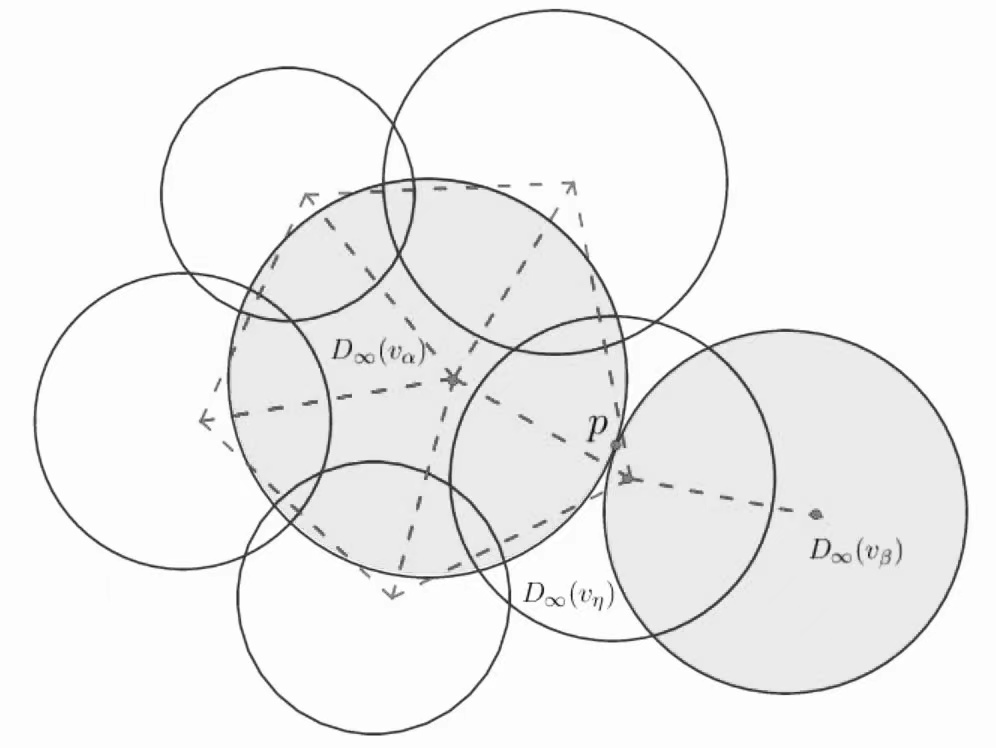}
\caption{If the disks $D_\infty(v_\alpha),D_\infty(v_\beta)$ touch at a point, Lemma~\ref{L-4-4} implies the overlap angle function violates condition $\mathrm{\mathbf{(c1)}}$}
\end{figure}

\begin{itemize}
\item[$(i)$] $D_{\infty}(v_\alpha)\cap D_{\infty}(v_\beta)$ consists of a unique point $p$. Recall that $\mathcal P_\infty$ is a \textbf{$T$-type} circle pattern with overlap angle function $\Theta_\infty$ satisfying condition $\mathrm{\mathbf{(c2)}}$. Due to Lemma~\ref{L-4-7}, there exists $v_\eta\in V\setminus\{v_\alpha,v_\beta\}$ such that
\[
p\in\mathbb D_{\infty}(v_\eta),
\]
which implies both $D_{\infty}(v_\alpha)\cap D_{\infty}(v_\eta)$ and $D_{\infty}(v_\beta)\cap D_{\infty}(v_\eta)$ contain interior points. We claim $v_\eta$ is an adjacent vertex of both $v_\alpha$ and $v_\beta$. Otherwise, either $D_{\infty}(v_\alpha), D_{\infty}(v_\eta)$ or $D_{\infty}(v_\beta), D_{\infty}(v_\eta)$ form a pair of externally tangent disks, which leads to contradictions. Notice that
    \[
D_{\infty}(v_\alpha)\cap D_{\infty}(v_\beta)=\{p\}\subset\mathbb D_{\infty}(v_\eta).
    \]
Lemma~\ref{L-4-4} gives
\[
\Theta_\infty([v_\alpha,v_\eta])+\Theta_\infty([v_\eta,v_\beta])>\pi,
\]
which violates condition $\mathrm{\mathbf{(c1)}}$.
\item[$(ii)$] $D_{\infty}(v_\alpha)\cap D_{\infty}(v_\beta)=\partial D_{\infty}(v_\alpha)=\partial D_{\infty}(v_\beta)$. That means $\mathbb D_{\infty}(v_\beta)$ is the complement of $ D_{\infty}(v_\alpha)$. In case that every $v_\tau\in V\setminus\{v_\alpha,v_\beta\}$ is an adjacent vertex of  both $v_\alpha$ and $v_\beta$, then $T$ is the boundary of an $m$-gonal bipyramid. Apparently,
\[
\Theta_\infty([v_\alpha,v_\tau])+\Theta_\infty([v_\tau,v_\beta])=\pi.
\]
If $m=3$, then $T$ is the boundary of a double tetrahedron, which means the above formula  violates  condition $\mathrm{\mathbf{(c1)}}$. If $m\geq 4$, we can find edges $e_1,e_2,e_3,e_4$  forming a simple closed curve separating the vertices of $T$ such that
\[
\sum\nolimits_{\mu=1}^4 \Theta(e_\mu)= 2\pi,
\]
which violates condition $\mathrm{\mathbf{(c4)}}$.

It remains to consider the case that there exists $v_{\tau_0}\in V\setminus\{v_\alpha,v_\beta\}$ which is non-adjacent to $v_\alpha$ or $v_\beta$. Without loss of generality, suppose $v_{\tau_0}$ is not adjacent to $v_\alpha$. Then the disks $D_{\infty}(v_{\tau_0}), D_{\infty}(v_\alpha)$ are externally tangent, which implies
\[
\mathbb D_{\infty}(v_{\tau_0})\subset\hat{\mathbb C}\setminus D_{\infty}(v_\alpha)=\mathbb D_{\infty}(v_\beta).
\]
As a result,
\[
\mathbb D_{\infty}(v_{\tau_0})\cap\mathbb D_{\infty}(v_\beta)=\mathbb D_{\infty}(v_{\tau_0}).
\]
On the other hand, when  $v_{\tau_0},v_\beta$ are non-adjacent vertices, we have
\[
\mathbb D_{\infty}(v_{\tau_0})\cap\mathbb D_{\infty}(v_\beta)=\emptyset.
\]
When $v_{\tau_0},v_\beta$ are adjacent vertices, we know $\Theta_\infty([v_{\tau_0},v_\beta])\in(0,\pi)$, which yields
\[
\mathbb D_{\infty}(v_{\tau_0})\cap\mathbb D_{\infty}(v_\beta)\subsetneq \mathbb D_{\infty}(v_{\tau_0}).
\]
No matter which case occurs, this leads to contradictions.
\end{itemize}

We thus prove $\mathcal P_\infty$ has contact graph isomorphic to the $1$-skeleton of $T$. To check $\mathcal P_\infty$ is \textbf{irreducible}, we need to show
\[
\cup_{w\in A} D_{\infty}(w)\subsetneq\hat{\mathbb C}
\]
for each strict subset $A$ of $V$. Recall that $r(D_\infty(v))\leq \pi/2$ for all $v\in V$. By Lemma~\ref{L-4-3}, we assert $\mathcal P_\infty$ satisfies condition $\mathrm{\mathbf{(cr)}}$. Choose a triangle of $T$ with vertices $v_i,v_j\in V$, $v_k\in V\setminus A$. For each $w\in V\setminus\{v_i,v_j,v_k\}$, it follows from Lemma~\ref{L-3-2} that
\[
D_{\infty}(v_i)\cap D_{\infty}(v_j)\cap D_{\infty}(v_k)\cap D_{\infty}(w)=\emptyset.
\]
In particular, the point $q=\partial D_{\infty}(v_i)\cap\partial D_{\infty}(v_j)\cap D_{\infty}(v_k)$ does not belong to any disks in $\mathcal P_\infty$ except for
$D_{\infty}(v_i), D_{\infty}(v_j),D_{\infty}(v_k)$. Hence there is a neighborhood $N_q$ of $q$ such that
\[
N_q\cap D_\infty(w)=\emptyset,\quad\forall\,w\in V\setminus\{v_i,v_j,v_k\}.
\]
Setting $O_q=N_q\setminus(D_{\infty}(v_i)\cup D_{\infty}(v_j))$, we obtain
\[
\big(\cup_{w\in A}D_{\infty}(w)\big)\cap O_q=\emptyset.
\]
Under condition $\mathrm{\mathbf{(c2)}}$, it is easy to see $O_q\neq\emptyset$. Therefore,
\[
\cup_{w\in A} D_{\infty}(w)\subset\hat{\mathbb C}\setminus O_q  \subsetneq\hat{\mathbb C},
\]
which concludes $\mathcal P_\infty$ is \textbf{irreducible}. The proof is completed.
\end{proof}

Combining Theorem~\ref{T-1-5} and Theorem~\ref{T-3-1}, we have the following Normal Family Theorem for hyperbolic polyhedra. Let $\Lambda$ be the set of dihedral angle functions $\Theta:\mathcal E\to(0,\pi)$ satisfying conditions $\mathrm{\mathbf{(a1)}}-\mathrm{\mathbf{(a4)}}$.
\begin{theorem}\label{T-5-1}
Let $\{Q_n\}$ be a sequence of compact convex polyhedra in $\mathbb H^3$, each of which is combinatorially equivalent to an abstract trivalent polyhedron $P$ with more than four faces. Then $\{Q_n\}$ (modulo isometries of $\mathbb H^3$) contains a subsequence convergent to a compact convex polyhedron in $\mathbb H^3$ of the same combinatorial type if the corresponding dihedral angle function sequence $\{\Theta_n\}$ stays in a compact subset of $\Lambda$.
\end{theorem}

\begin{remark}
A more direct approach to this theorem might be following the works of Andreev~\cite{Andreev}, Rivin-Hodgson~\cite{Rivin-Hodgson} and Roeder-Hubbard-Dunbar~\cite{Roeder-Hubbard-Dunbar}. In particular, most arguments in~\cite{Andreev,Roeder-Hubbard-Dunbar} remain valid in our situation. The only substantial modification is to check that all faces of the limit polyhedron are still parallelograms. Toward this end, Andreev~\cite{Andreev} and Roeder-Hubbard-Dunbar~\cite{Roeder-Hubbard-Dunbar} made the use of Gauss-Bonnet formula, while our observations (Lemma~\ref{L-4-4} and Lemma~\ref{L-4-7}) indicate that condition $\mathrm{\mathbf{(a1)}}$ is approximately adequate for this purpose. %Finally, we mention that Theorem~\ref{T-5-1} can provide an alternative proof of the general Andreev's Theorem.
\end{remark}

\section{Some questions}\label{S-6a}
The paper leaves open a number of questions. We think the following are some particularly interesting subjects for further developments.

\medskip
1. \textbf{Further generalization of Andreev's Theorem}. As before, let $P$ be an abstract trivalent polyhedron with $\mathcal E$ being the set of edges. We say two edges $e_i,e_j$ of $P$ form a \textbf{Whitehead pair} if $e_i,e_j$ belong to a common face of $P$ and all of their endpoints are distinct. Precisely, we ask whether the following conjecture holds. For simplicity, here we do not include the case that $P$ is the triangular prism.
\begin{conjecture}
Let $P$ be an abstract trivalent polyhedron with more than five faces. Assume that $\Theta:\mathcal{E}\to (0,\pi)$ is a function satisfying the  conditions below:
\begin{itemize}
\item[$\mathbf{(s1)}$] Whenever $\Gamma$ is a curve forming the boundary of the union of two adjacent triangles of the dual complex $P^\ast$ and intersecting edges $e_1,e_2,e_3,e_4$ such that both $e_1,e_2$ and $e_3,e_4$ form \textbf{Whitehead pairs},   then either $\Theta(e_1)+\Theta(e_2)\leq \pi$ or $\Theta(e_3)+\Theta(e_4)\leq \pi$.
\item[$\mathbf{(s2)}$] Whenever three distinct edges $e_1,e_2,e_3$ of $P$ meet at a vertex, then $\sum_{\mu=1}^3\Theta(e_\mu)>\pi$,  and
    $\Theta(e_1)+\Theta(e_2)<\Theta(e_3)+\pi$, $\Theta(e_2)+\Theta(e_3)<\Theta(e_1)+\pi$, $\Theta(e_3)+\Theta(e_1)<\Theta(e_2)+\pi$.
\item[$\mathbf{(s3)}$] Whenever $\Gamma$ is a \textbf{prismatic k-circuit} intersecting edges $e_1,e_2,\cdots,e_k$, then $\sum_{\mu=1}^k\Theta(e_\mu)<(k-2)\pi$.
\end{itemize}
Then there exists a compact convex hyperbolic polyhedron $Q$ combinatorially equivalent to $P$ with dihedral angles given by $\Theta$. Furthermore, $Q$ is unique up to isometries of $\mathbb H^3$.
\end{conjecture}

We believe it a promising conjecture based on the following reasons:
\begin{itemize}
\item{} It provides a unified generalization of Theorem~\ref{T-1-3} and Bao-Bonahon's Theorem on hyperideal polyhedra~\cite{Bao-Bonahon}.
\item{} The rigidity part has been proved by Rivin-Hodgson~\cite[Corollary 4.6]{Rivin-Hodgson}.
\item{} Under conditions $\mathbf{(s1)}-\mathbf{(s3)}$, one can show a "weak solution" for the corresponding circle pattern problem exists.
\end{itemize}

On the other hand, it might be much more difficult to further extend Theorem~\ref{T-1-1} once condition $\mathbf{(c1)}$ is relaxed. Actually, from the proof of Theorem~\ref{T-3-1}, we can see the requirement of avoiding extraneous overlaps is somewhat "excessively adequate" for determining combinatorics of the corresponding polyhedra.

\medskip
2. \textbf{Developing algorithms}. In cases of non-obtuse angles, there are many successful computational methods to construct the desired circle patterns, indicated in the works of Collins-Stephenson~\cite{Collins-Stephenson}, Chow-Luo~\cite{Chow-Luo}, Orick-Stephenson-Collins~\cite{Orick-Stephenson-Collins}, Connelly-Gortler~\cite{Connelly-Gortler}, Bowers~\cite{Bowers} and others.
We are looking forward to more numerical experiments to check whether these algorithms still work well under conditions of Theorem~\ref{T-1-1}. Moreover, can one prove the convergence?

%We may pursue the question of generalizing Theorem~\ref{T-1-4}. Nonetheless, it is much more difficult to control combinatorial structure of the contact graph once the conditions of Theorem~\ref{T-1-4} are further relaxed. In fact, using Rivin's Theorem~\cite{Rivin}, we find that convex hyperbolic polyhedral in a fixed combinatorial class may provide circle patterns with non-isomorphic contact graphs.

\section{Appendix}\label{S-7a}
This section is devoted to proofs of Lemma~\ref{L-4-1} and Lemma~\ref{L-4-6}.

\begin{proof}[\textbf{Proof of Lemma~\ref{L-4-1}}]
Set
\begin{equation}\label{E-6-1}
l_{ij}=\arccos(\cos r_{i}\cos r_{j}-\cos\Theta_{ij}\sin r_{i}\sin r_{j})
\end{equation}
and $l_{jk},l_{ki}$ similarly. We reduce the proof to showing the following  inequalities:
\[
l_{ij}+l_{jk}>l_{ki},\;\;\; l_{jk}+l_{ki}>l_{ij},\;\;\; l_{ki}+l_{ij}>l_{jk}, \;\;\; l_{ij}+l_{jk}+l_{ki}<2\pi.
\]
Equivalently, let us verify
\begin{equation}\label{E-6-2}
\begin{aligned}
&4\sin\dfrac{l_{ij}+l_{jk}+l_{ki}}{2}\sin\dfrac{l_{ij}+l_{jk}-l_{ki}}{2}
\sin\dfrac{l_{jk}+l_{ki}-l_{ij}}{2}\sin\dfrac{l_{ki}+l_{ij}-l_{jk}}{2} \\
=&\sin^2 l_{ij}\sin^2 l_{jk}-(\cos l_{ij}\cos l_{jk}-\cos l_{ki})^2>0.
\end{aligned}
\end{equation}
To simplify notations, for $\mu=i,j,k$, set $a_\mu=\cos r_\mu$, $x_\mu=\sin r_\mu$.
Combining~\eqref{E-6-1} and~\eqref{E-6-2}, we need to demonstrate
\begin{equation}\label{E-6-3}
\begin{aligned}
&\sin^2\Theta_{ij}x^2_ix^2_j+\sin^2\Theta_{jk}x_j^2x^2_k+\sin^2\Theta_{ki}x^2_k x_i^2-(2+2\cos\Theta_{ij}\cos\Theta_{jk}\cos\Theta_{ki})x_i^2x_j^2x_k^2\\
&\;\;+2\lambda_{jk}^ia_ja_kx_jx_kx_i^2+2\lambda_{ki}^ja_ka_ix_kx_ix_j^2+2\lambda_{ij}^ka_ia_jx_ix_jx_k^2>0,
\end{aligned}
\end{equation}
where
\[
\lambda_{jk}^i=\cos\Theta_{jk}+\cos\Theta_{ki}\cos\Theta_{ij}.
\]
Since $a^2_\mu+x^2_\mu=1$, an equivalent form of~\eqref{E-6-3} is
\begin{equation}\label{E-6-4}
\begin{split}
&\sin^2\Theta_{ij}a_k^2x^2_ix^2_j+\sin^2\Theta_{jk}a_i^2x_j^2x^2_k+\sin^2\Theta_{ki}a_j^2x^2_k x_i^2
+\zeta_{ijk}x_i^2x_j^2x_k^2\\
&\quad+2\lambda^i_{jk}a_ja_kx_jx_kx_i^2+2\lambda^j_{ki}a_ka_ix_kx_ix_j^2+2\lambda^k_{ij}a_ia_jx_ix_jx_k^2>0,
\end{split}
\end{equation}
where
\[
\begin{aligned}
\zeta_{ijk}=&
1-\cos^2\Theta_{ij}-\cos^2\Theta_{jk}-\cos^2\Theta_{ki}
-2\cos\Theta_{ij}\cos\Theta_{jk}\cos\Theta_{ki}\\
=&-4 \cos\dfrac{\Theta_{ij}+\Theta_{jk}+\Theta_{ki}}{2}
\cos\dfrac{\Theta_{ij}+\Theta_{jk}-\Theta_{ki}}{2}
\cos\dfrac{\Theta_{jk}+\Theta_{ki}-\Theta_{ij}}{2}
\cos\dfrac{\Theta_{ki}+\Theta_{ij}-\Theta_{jk}}{2}\\
>&0.
\end{aligned}
\]
Notice that $\Theta_{jk},\Theta_{ki},\Theta_{ij}$ form the inner angles of
a spherical triangle. Let $\phi_i,\phi_j,\phi_k$ denote the lengths of sides opposite to $\Theta_{jk},\Theta_{ki},\Theta_{ij}$, respectively. The second spherical law of cosines gives
\begin{equation}\label{E-6-5}
\lambda^i_{jk}=\cos\Theta_{jk}+\cos\Theta_{ki}\cos\Theta_{ij}= \cos\phi_i\sin\Theta_{ki}\sin\Theta_{ij}.
\end{equation}
Set
\[
s_{i}=\sin\Theta_{jk}a_ix_jx_k,\;\;\; s_{j}=\sin\Theta_{ki}a_jx_kx_i,\;\;\;   s_{k}=\sin\Theta_{ij}a_kx_ix_j.
\]
Inserting~\eqref{E-6-5} into~\eqref{E-6-4}, we reduce the proof to showing
\[
s_{i}^2+s_{j}^2+s_{k}^2+2\cos \phi_is_{j}s_{k}+2\cos \phi_j s_{k}s_{i}+2\cos \phi_ks_{i}s_{j}+\zeta_{ijk}x_i^2x_j^2x_k^2>0.
\]
For this purpose, we write the left side of the formula as the sum of $\zeta_{ijk}x_i^2x_j^2x_k^2$ and the quadratic form in variables $s_i,s_j,s_k$ with the matrix
\[
A=\left[\begin{array}{ccc}
1 & \cos \phi_k & \cos \phi_j \\
\cos \phi_k & 1 & \cos \phi_i \\
\cos \phi_j & \cos \phi_i & 1
\end{array}\right].
\]
Recalling $\phi_i,\phi_j,\phi_k$ are lengths of a spherical triangle, we obtain
\[
\operatorname{det}(A)=4 \sin\dfrac{\phi_i+\phi_j+\phi_k}{2}
\sin\dfrac{\phi_i+\phi_j-\phi_k}{2}
\sin\dfrac{\phi_j+\phi_k-\phi_i}{2}
\sin\dfrac{\phi_k+\phi_i-\phi_j}{2}>0.
\]
By Sylvester's criterion, we then check $A$ is positive-definite.
As a result,
\[
s_{i}^2+s_{j}^2+s_{k}^2+2\cos \phi_is_{j}s_{k}+2\cos \phi_j s_{k}s_{i}+2\cos \phi_ks_{i}s_{j}+\zeta_{ijk}x_i^2x_j^2x_k^2>0.
\]
We thus finish the proof.
\end{proof}

\begin{comment}
%&
%1-\cos^2\Theta_{ij}-\cos^2\Theta_{jk}-\cos^2\Theta_{ki}
%-2\cos\Theta_{ij}\cos\Theta_{jk}\cos\Theta_{ki}\\
completing the square gives
\[ Meanwhile, a routine computation gives
\[
\zeta_{ijk}=
>0.
\]
\begin{aligned}
& s_{i}^2+s_{j}^2+s_{k}^2+2\cos \phi_is_{j}s_{k}+
2\cos \phi_j s_{k}s_{i}+2\cos \phi_ks_{i}s_{j}\\
=&(s_{i}+\cos\phi_ks_{j}+\cos \phi_js_{k})^2
+\sin^2\phi_ks^2_{j}+\sin^2\phi_js^2_{k}
+2(\cos\phi_i-\cos\phi_j\cos\phi_k)s_{j}s_{k}\\
\geq & \sin^2\phi_ks^2_{j}+\sin^2\phi_js^2_{k}+
2(\cos\phi_i-\cos\phi_j\cos\phi_k)s_{j}s_{k}.
\end{aligned}
\]
Using the spherical law of cosines, we obtain
\[
\cos\phi_i-\cos\phi_j\cos\phi_k=\cos\Theta_{jk}\sin\phi_j\sin\phi_k.
\]
It follows that
\[
\begin{aligned}
& s_{i}^2+s_{j}^2+s_{k}^2+2\cos \phi_is_{j}s_{k}+
2\cos \phi_j s_{k}s_{i}+2\cos \phi_ks_{i}s_{j}\\
\geq &\sin^2\phi_ks^2_{j}+\sin^2\phi_js^2_{k}+
2\cos\Theta_{jk}\sin\phi_j\sin\phi_ks_{j}s_{k}\\
=&(\sin\phi_ks_{j}+\cos\Theta_{jk}\sin\phi_js_{k})^2
+\sin^2\Theta_{jk}\sin^2\phi_js^2_{k}\\
\geq &\sin^2\phi_j\sin^2\Theta_{ij}\sin^2\Theta_{jk}a_k^2x_i^2x_j^2,
\end{aligned}
\]
Meanwhile, a routine computation yields
\[
\begin{aligned}
\zeta_{ijk}&=1-\cos^2\Theta_{jk}-\cos^2\Theta_{ki}-\cos^2\Theta_{ij}
-2\cos\Theta_{jk}\cos\Theta_{ki}\cos\Theta_{ij}\\
         &=\sin^2\Theta_{ij}\sin^2\Theta_{jk}-(\cos\Theta_{ki}+\cos\Theta_{ij}\cos\Theta_{jk})^2\\
        &=\sin^2\Theta_{ij}\sin^2\Theta_{jk}-\cos^2\phi_j\sin^2\Theta_{ij}\sin^2\Theta_{jk}\\
         &=\sin^2\phi_j\sin^2\Theta_{ij}\sin^2\Theta_{jk}.\\
\end{aligned}
\]
\end{comment}

\begin{proof}[\textbf{Proof of Lemma~\ref{L-4-6}}]
We follow the idea of He~\cite[Lemma 7.1]{He}. Assume the lemma is not true. For each positive integer $n$, then there exists a circle pattern $\mathcal P_n=\{D_n(v)\}_{v\in V}$ on $(\hat{\mathbb C},\mathrm {ds})$
with contact graph isomorphic to the $1$-skeleton of $T$ and overlap angle function $\Theta_n\in\Omega$ such that
\[
r(D_n(v_\alpha))> n \times r(D_n(u_\beta))
\]
for a pair of adjacent vertices $v_\alpha,u_\beta\in V$. Moreover, $\mathcal P_n$ is subject to the requirement that $r(D_n(v))\leq \pi/2$ for all $v\in V$. By passing to a subsequence if necessary, we assume for each vertex $v\in V$, $\lim_{n\to\infty}r(D_n(v))$ exists in $[0,\pi/2]$.

Let $V^0$ be the set of vertices $v\in V$ for which $\lim_{n\to\infty}r(D_n(v))=0$ and let $V^+=V\setminus V^0$. Obviously, we have $u_\beta\in V^0$. Furthermore, we claim $V^+\neq\emptyset$. Otherwise, $r(D_n(v))\to 0$ for all $v\in V$. It follows that the diameter of $(\hat{\mathbb C},\mathrm {ds})\cong\mathbb S^2$ tends to zero, which leads to a contradiction. Since  $V^+, V^0$ form a pair of complementary, nonempty subsets of $V$, there exist two adjacent vertices belonging to $V^+$ and $V^0$, respectively. Without loss of generality, we assume $v_\alpha\in V^+$, $u_\beta\in V^0$.

Passing to a subsequence again, we assume further the disk $D_n(v)$ converges to some disk $D_{\infty}(v)$ in $(\hat{\mathbb C},\mathrm {ds})$ for each
 $v\in V^+$. Meanwhile, we use $\Theta_{\infty}([u,w])$ to denote the overlap angle between the disks $D_{\infty}(u),D_{\infty}(w)$ provided $u,w\in V^+$ are adjacent vertices. Since $\{\Theta_n\}$ stays in the compact subset $\Omega\subset W$, we derive $\Theta_\infty$ satisfies conditions $\mathrm{\mathbf{(c1)}}-\mathrm{\mathbf{(c4)}}$ when the related overlap angles exist.

Notice that there exist at least two vertices in $V^+$ adjacent to $v_\alpha$. Let $u_{1},\cdots, u_{d}$, in anticlockwise order, be all adjacent vertices of $v_\alpha$. Without loss of generality, we may assume $u_1\in V^+$, $u_2=u_\beta\in V_0$. Set
\[
j=\min\{\,l\,|\,l>2,\,u_l\in V^+\,\}.
\]
Then $u_j\in V^{+}$ and $u_i\in V^0$ for $i=2,\cdots,j-1$. One is ready to see the circles $\partial D_\infty(u_1),\partial D_\infty(u_j)$ have an intersection point in $\partial D_{\infty}(v_\alpha)$, which yields
\begin{equation}\label{E-6-6}
\partial D_\infty(u_1)\cap \partial D_\infty(u_j)\cap\partial D_{\infty}(v_\alpha)\neq \emptyset.
\end{equation}
We claim $u_1,u_j$ are non-adjacent. Otherwise, either $v_\alpha, u_1,u_j$ are the vertices of a triangle of $T$ or the edges $[v_\alpha,u_1], [u_1,u_j], [u_j,v_\alpha]$ form a simple closed curve separating the vertices of $T$. In the first case, recall $\Theta_\infty$ satisfies conditions $\mathrm{\mathbf{(c2)}}$. Lemma~\ref{L-4-2} gives
\[
\partial D_\infty(u_1)\cap \partial D_\infty(u_j)\cap\partial D_{\infty}(v_\alpha)=\emptyset,
\]%.
which leads to a contradiction. In the second case, note that $r(D_{\infty}(v))\leq \pi/2$ for all $v\in V$. Because $\Theta_\infty$ satisfies conditions $\mathrm{\mathbf{(c3)}}$, it follows from Lemma~\ref{L-4-3} that
\[
D_\infty(u_1)\cap D_\infty(u_j)\cap D_{\infty}(v_\alpha)=\emptyset,
\]
which also contradicts~\eqref{E-6-6}.

We thus assert $u_1,u_j\in V^{+}$ are non-adjacent vertices. As $r(D_n(u_i))\to 0$ for $i=2,\cdots,j-1$, the disks $D_\infty(u_1),D_\infty(u_j)$ are externally tangent. Therefore, either $D_\infty(u_1),D_\infty(u_j)$ intersect at a unique point $p\in \partial D_\infty(v_\alpha)$ or $D_\infty(u_1), D_\infty(u_j)$ share the boundary.
In the first case, Lemma~\ref{L-4-4} indicates
\begin{equation}\label{E-6-7}
\Theta_{\infty}([v_\alpha,u_1])+\Theta_{\infty}([u_j,v_\alpha])=\pi.
\end{equation}
In the second case, one also derives~\eqref{E-6-7}

Meanwhile, let $w_1=v_\alpha,w_2=u_{j-1},\cdots, w_{m}$, in anticlockwise order, be all adjacent vertices of $u_j$. As before, we obtain $w_k\in V^+$ non-adjacent to $w_1=v_\alpha$ such that $w_i\in V^0$ for $i=2,\cdots,k-1$. We claim $w_k\neq u_1$. Otherwise, $w_k, v_\alpha$ are adjacent vertices, which leads to a contradiction.

It follows that the vertices $u_1, w_k$ are either adjacent or non-adjacent. In the first case, similar arguments to formula~\eqref{E-6-7} give
\begin{equation}\label{E-6-8}
\Theta_{\infty}([u_1, w_k])+\Theta_{\infty}([w_k,u_j])=\pi.
\end{equation}
Combing~\eqref{E-6-7} and~\eqref{E-6-8}, we obtain
\[
\Theta_{\infty}([v_\alpha,u_1])+\Theta_{\infty}([u_1, w_k])
+\Theta_{\infty}([w_k,u_j])+\Theta_{\infty}([u_j,v_\alpha])=2\pi,
\]
which violates condition $\mathrm{\mathbf{(c4)}}$, considering that the edges
$[v_\alpha,u_1],[u_1, w_k],[w_k,u_j], [u_j,v_\alpha]$ form a simple closed curve separating the vertices of $T$. In the second case, it is easy to see both the disks $D_\infty(u_j), D_\infty(w_k)$ are externally tangent to $D_\infty(u_1)$ at a common point, which implies
\[
\Theta_{\infty}([u_j,w_k])=\pi.
\]
This contradicts that $\{\Theta_n\}$ stays in a compact subset of $W$. Thus the lemma is proved.
\end{proof}

\bigskip
\noindent \textbf{Acknowledgements}.
The author is supported by NSFC (No.12371075 and No.11631010). We would like to thank Steven Gortler for a great many helpful conversations and encouragements. We also thank both the Institute for Computational and Experimental Research Mathematics (ICERM) at Brown University and the organizers for hosting the workshop on circle packings and geometric rigidity that indirectly led to the development of this paper.

%chool of Mathematics, Hunan University, Changsha 410082, P.R. China
%\bigskip
%\indent SCHOOL OF MATHEMATICS, HUNAN UNIVERSITY, CHANGSHA 410000, P.R. CHINA \\ %\\[2pt]
%\indent E-mail address: zhouze@hnu.edu.cn
\end{document}